\documentclass[12]{article}

\RequirePackage[OT1]{fontenc}
\RequirePackage{amsthm,amsmath,amssymb,amsfonts,graphicx,setspace}
\RequirePackage[numbers]{natbib}
\RequirePackage[colorlinks,citecolor=blue,urlcolor=blue]{hyperref}

\numberwithin{equation}{section}
\theoremstyle{plain}
\newtheorem{theorem}{Theorem}
\newtheorem{corollary}{Corollary}
\newtheorem{lemma}{Lemma}

\def\P{\mathbb P}

\def\E{\mathbb E}

\newcommand{\bX}{\mbox{$\bf X$}}
\newcommand{\blds}[1]{\mbox{\scriptsize \bf #1}}

\newcommand{\tbX}{\mbox{$T({\bf X})$}}

\newcommand{\bY}{\mbox{$\bf Y$}}

\newcommand{\bx}{\mbox{$\bf x$}}
\newcommand{\br}{\mbox{$\bf r$}}
\newcommand{\bs}{\mbox{$\bf s$}}
\newcommand{\bk}{\mbox{$\bf k$}}

\def\real{\mathbb R}
\newcommand{\deq}{\stackrel{\scriptscriptstyle\triangle}{=}}
\begin{document}

\title{Consistent Testing for Recurrent Genomic Aberrations}

\author{Vonn Walter \\
	Lineberger Comprehensive Cancer Center, \\
	University of North Carolina at Chapel Hill, Chapel Hill, NC 27579 \\
	\and
	Fred A. Wright \\
	Departments of Statistics and Biological Sciences \\
          and the Bioinformatics Research Center, \\
	North Carolina State University, Raleigh, NC 27695
	\and
	Andrew B. Nobel \\
	Department of Statistics and Operations Research, \\
	Department of Biostatistics, and Lineberger Comprehensive Cancer Center \\
	University of North Carolina at Chapel Hill, Chapel Hill, NC 27579
	}
\maketitle

\newpage
\begin{abstract}
Genomic aberrations, such as somatic copy number alterations, are frequently observed in tumor tissue.   Recurrent aberrations, occurring in the
same region across multiple subjects, are of interest because they may highlight genes associated with tumor development or progression.
A number of tools have been proposed to assess the statistical significance of recurrent DNA copy number aberrations, but their statistical properties
have not been carefully studied.  Cyclic shift testing, a permutation procedure using independent random shifts of genomic marker observations on the genome,
has been proposed to identify recurrent aberrations, and is potentially useful for a wider variety of purposes, including identifying regions with methylation
aberrations or overrepresented in disease association studies. For data following a countable-state Markov model, we prove the asymptotic validity of cyclic shift $p$-values under a fixed sample size regime as the number of observed markers
tends to infinity.  We illustrate cyclic shift testing for a variety of data types, producing biologically relevant findings for three publicly available datasets.
\end{abstract}

\section{Introduction}

Many genomic datasets consist of measurements from multiple samples at a common set of 
genetic markers, with no ``phenotype" representing clinical state or experimental condition of the sample.
Datasets of this type include genome-wide measurements of DNA copy number 
or DNA methylation, for which the main goal is to identify aberrant regions on the genome that tend to have extreme measurements
in comparison to other regions.  Testing for aberrations  requires some thought about appropriate test statistics, and 
constructing a null distribution that appropriately reflects serial correlation structures inherent to genomic data.
A meta-analysis across several genome-wide association studies might also be viewed in this framework, in the sense that the testing
for association within each study produces a vector of $p$-values that might be viewed as a vector of ``observations."

The problem of interest to us is the identification of aberrant markers, where multiple samples 
exhibit a coordinated (unidirectional), departure from the expected state.  Aberrant markers are of
particular interest in cancer studies, where tumor suppressors or oncogenes exhibit DNA copy variation or modified methylation levels.  
Similarly, it may be possible to identify pleiotropic single nucleotide polymorphisms (SNPs) in disease association by identifying 
genetic markers that repeatedly give rise to small $p$-values in multiple association studies.

In this paper we provide a rigorous asymptotic analysis of a permutation based testing procedure 
for identifying aberrant markers in genomic data sets.  The procedure, called DiNAMIC, was introduced in 
Walter et al.\ (2011), \nocite{walter2011} and is described in detail below.  In contrast to other procedures
which permute all observations, DiNAMIC is based on cyclic shifting of samples.  Cyclic shifting eliminates
concurrent findings across samples, but retains the 
adjacency of observations in a sample (with the exception of the first and last entries), thereby largely preserving 
the correlation structure among markers.  Our principal result is that, for a broad family of null
data distributions, the sampling distribution of the DiNAMIC
procedure is close to the true conditional distribution of the data restricted to its cyclic shifts.  As a corollary,
we find that the cyclic shift testing provides asymptotically correct Type I error rates.

The outline of the paper is as follows.  The next section is devoted to a description of the cyclic shift procedure, 
a discussion of the underlying testing framework within which our analysis is carried out, and a statement of our 
principal result.  In Section \ref{sec3} we apply cyclic shift testing to DNA copy number analysis, 
DNA methylation analysis, and meta-analysis of GWAS data, and show that the results are  
consistent with the existing biological literature.  Because of its broad applicability and 
solid statistical foundation, we believe that cyclic shift testing is a valuable tool for the identification
of aberrant markers in many large scale genomic studies.

\section{Asymptotic Consistency of Cyclic Shift Permutation}

\subsection{Data Matrix}

We consider a data set derived from $n$ subjects at $m$ common genomic locations or markers.  The data is arranged in an 
$n \times m$ matrix $\bX$ with values in a set $\cal{A} \subseteq \real$.  Depending on the application, $\cal{A}$ may be
finite or infinite.
The entry $x_{ij}$ of $\bX$ contains data from subject $i$ at marker $j$.  Thus the $i$th row $\bX_{i \cdot}$ 
of $\bX$ contains the data from subject $i$ at all markers, and the $j$th column 
$\bX_{\cdot j}$ of $\bX$ contains the data at marker $j$ across subjects.  
For $1 \leq j \leq m$ let $s_j = s_j(\bX_{\cdot j})$ be a local summary statistic for 
the $j$th marker.  
In most applications the simple sum statistic $s_j = \sum_{i = 1}^n x_{ij}$ is employed.  
In order to identify locations with coordinated departures from baseline behavior, we apply a global
summary statistic to the local statistics $s_1,\ldots, s_m$.  When looking for extreme, positive departures from
baseline it is natural to employ the global statistic
\begin{equation}
\label{globalstat}
\tbX \ = \ {\mathrm{max}} (s_1, \dots, s_m) .
\end{equation}
To detect negative departures from baseline, the maximum may be replaced by a minimum.    
The cyclic shift procedure and the supporting theory in Theorem \ref{thm1} apply to arbitrary local 
statistics, as well as a range of global statistics. 

\subsection{Cyclic Shift Testing}

Given a data matrix $\bX$, we are interested in assessing the significance of the observed value $t_0 = T(\bX)$ 
of the global statistic.  When $t_0$ is found to be significant, the identity and location of the marker $j$ having 
the maximum (or minimum) local statistic is of primary biological importance.  While in special cases 
it is possible to compute $p$-values for $t_0$ under parametric assumptions, permutation
based approaches are often an attractive and more flexible alternative.  A permutation based $p$-value can be
obtained by applying permutations $\pi$ to the entries of $\bX$, producing the matrices $\pi(\bX)$, and then comparing $t_0$ to the resulting values 
$T(\pi(\bX))$ of the global statistic.  
The maximum global statistic accounts for multiple comparisons across markers, so it
is not necessary to apply further multiplicity correction to the permuted values $T(\pi(\bX))$.

The performance and suitability of permutation based $p$-values in the marker identification problem
depends critically on the family of allowable permutations $\pi$.  If $\pi$ permutes the entries of $\bX$
without preserving row or column membership, then the induced null distribution is equivalent to sampling
the entries of $\bX$ at random without replacement.  In this case the induced null distribution does not capture 
the correlation of measurements within a sample, or systematic differences (e.g.\ in scale, 
location, correlation) between samples.  In real data, correlations within and systematic differences 
between samples can be present even in the absence of aberrant markers.
As such, $p$-values obtained under full permutation of $\bX$ will be sensitive to secondary features of the data
and may yields significant $p$-values even when no aberrant markers are present.  
An obvious improvement of full permutation is to separately permute
the values in each row (sample) of the data matrix.  This approach is used in the GISTIC procedure of Beroukhim et al.\ (2007).  While row-by-row permutation preserves some differences 
between rows, it eliminates correlations within rows (and correlation differences between rows), so that the
induced null distribution is again sensitive to secondary, correlation based features of the data that are not
related to the presence of aberrant markers.

The DiNAMIC cyclic shift testing procedure of Walter et al.\ (2011) \nocite{walter2011}  
addresses the shortcomings of full and row-by-row permutation by further restricting the 
set of allowable permutations.  In the procedure, each row of the data matrix 
is shifted to the left in a cyclic fashion, as detailed below, so that the first $k$ entries of the vector are placed 
after the last element; the values of the offsets $k$ are chosen independently from row
to row.  Cyclic shifting preserves the serial correlation structure with each sample, except at
the single break point where the last and first elements of the unshifted sample are placed next
to one another.  At the same time, the use of different offsets breaks 
concurrency among the samples, so that the resulting cyclic null distribution is appropriate for
testing the significance of $t_0 = T(\bX)$.

\subsection{Cyclic Shift Testing}

Formally, a {\textit{cyclic shift}} of index $k \in \{0,\ldots,m-1\}$ is a map $\sigma_k:  {\cal{A}}^m \rightarrow {\cal{A}}^m$ 
whose action is defined as follows:
\[
\sigma_k(x_{1}, x_{2}, \dots, x_{m}) = (x_{k+1}, x_{k+2}, \dots, x_{m}, x_{1}, \dots, x_{k}).
\]
\noindent
Given $\bk = (k_1,\ldots,k_n)$ with $k_i \in \{0, \ldots, m-1\}$, 
let $\sigma_{\blds k} = \sigma_{k_1} \otimes \cdots \otimes \sigma_{k_n}$ be the 
map from the set ${\cal{A}}^{n \times m}$ of data matrices to itself defined by applying 
$\sigma_{k_i}$ to the $i$th row of $\bX$, namely,
\[
\sigma_{\blds k}(\bX) \ = \ (\sigma_{k_1}(\bX_{1 \cdot}), \ldots, \sigma_{k_n}(\bX_{n \cdot}))^t
\]
The cyclic shift testing procedure of Walter et al.\ (2011) \nocite{walter2011} is as follows.  

\vskip.2in

\noindent
{\bf Cyclic shift procedure to assess the statistical significance of $\tbX$}

\begin{enumerate}

\item Let $\sigma^1(\cdot), \dots, \sigma^N(\cdot)$ be random cyclic shifts of the form 
$\sigma_{k_1} \otimes \cdots \otimes \sigma_{k_n}$, where $k_1,\ldots,k_n$ are independent
and each is chosen uniformly from $\{0,\ldots, m-1\}$.

\item Compute the values $T(\sigma^1(\bX)), \ldots, T(\sigma^N(\bX))$ of the global statistic $T$
at the random cyclic shifts of $\bX$.

\item Define the percentile-based $p$-value
\[
p(T(\bX)) = \max \left( N^{-1} \sum_{l = 1}^N I(T(\sigma^l (\bX)) \geq T(\bX)), 1/N \right).
\]
Here $I(A)$ is the indicator function of the event $A$.

\end{enumerate}

\subsection{Testing Framework}

We wish to assess the performance of the cyclic shift procedure within a formal testing framework. 
To this end, we regard the observed data matrix $\bX$ as an observation from a probability 
distribution $P_m$ on ${\cal A}^{n \times m}$, so that
for any (measurable) set $A \subseteq {\cal A}^{n \times m}$ the probability
$
P_m(A) = \P( \bX \in A ) .
$
As measurements derived from distinct samples are typically independent, we restrict our attention
to the family of measures ${\cal P}$ on ${\cal A}^{n \times m}$ under which the rows of $\bX$ are 
independent.  

Let ${\cal P}_0 \subseteq {\cal P}$ be the sub-family of ${\cal P}$ corresponding to 
the null hypothesis that $\bX$ has no atypical markers, {\em i.e.}, no markers exhibiting coordinated 
activity across samples.  One may define ${\cal P}_0$ in a variety of ways, but the simplest is 
to let ${\cal P}_0$ be the set of distributions $P_m \in {\cal P}$ such that
the rows of $\bX$ are stationary and ergodic under $P_m$; 
independence of the rows follows from the definition of ${\cal P}$.
Under ${\cal P}_0$ the columns of $\bX$ are stationary and ergodic, and the same is true of the local statistics
$s_j$, which are identically distributed and have constant mean and variance.  
Thus under ${\cal P}_0$ no marker is atypical in a strong distributional sense.

Our principal result shows that the $p$-value produced by the cyclic shift procedure 
is approximately consistent for distributions $P_m$ in a subfamily ${\cal P}^* \subseteq {\cal P}_0$.
The family ${\cal P}^*$ includes or approximates many distributions of practical interest,
including finite order Markov chains with discrete or continuous state spaces.  
In order to assess the consistency of the cyclic shift $p$-value 
we carefully define both the target and the induced distributions of the procedure.
As much of what follows concerns probabilities conditional on the observed data matrix, 
we use $\bX$ to denote both the random matrix and its observed realization. 
Given $\bX$ let 
\[
{\cal S}_m(\bX) \ = \ \{ \sigma_{\blds k} (\bX) : \bk \in \{0,\ldots,m-1\}^n \} \ \subseteq \ {\cal A}^{n \times m} 
\]
be the set of all cyclic shifts of $\bX$.
Define the {\em true conditional distribution} to be the conditional
distribution of $P_m$ given ${\cal S}_m(\bX)$, namely
\[
P_{\blds X} (A) \ = \ P_m( A \, | \, {\cal S}_m(\bX)) \ \ \ \ A \subseteq {\cal A}^{n \times m} .
\]
If $P_m$ is discrete with probability mass function $p(\cdot)$ then
\[
P_{\blds X} ( A ) \ = \ 
\frac{1}{\sum_{{\bf Y}' \in {\cal S}_m ({\bf X})} p(\bY')}
\sum_{\blds Y \in A}  p(\bY) \cdot I(\bY \in {\cal S}_m(\bX)) .
\]
If $P_m$ has probability density function $f(\cdot)$ then $P_{\blds X}$ may be defined in a similar fashion.

In the cyclic shift procedure, matrices are selected uniformly at random 
from the set ${\cal S}_m(\bX)$ of cyclic shifts of the observed data matrix $\bX$.  
The associated {\textit{cyclic conditional distribution}} has the form
\[
Q_{\blds X} (A)
\ =\ 
\sum_{\blds Y \in A} \frac{1}{\vert {\cal S}_m(\bX) \vert} \cdot I(\bY \in A) \ \ \ A \subseteq {\cal A}^{n \times m}.
\]
Under mild conditions the $m^n$ cyclic shifts 
of $\bX$ are distinct with high probability when $m$ is large (see Lemma \ref{lem3} in Section \ref{sec4}).  In this
case, the cyclic conditional distribution may be written as
\[
Q_{\blds X} ( A ) 
\ = \ 
\sum_{\blds Y \in A}  \frac{1}{m^n} \cdot I(\bY \in {\cal S}_m(\bX)) 
\ = \ 
\frac{1}{m^n} \cdot |A \cap {\cal S}_m(\bX)| \ \ \ \ A \subseteq {\cal A}^{n \times m} .
\]
The distribution of the cyclic shift $p$-value is given by
\[
p(T(\bX)) \sim \max( N^{-1} \, \mbox{Bin}(N,\alpha), 1/N) .
\]
Here $\alpha = Q_{\blds X} ( T \geq t_0 )$
where $t_0$ is the observed value of $T(\bX)$, and $T \geq t_0$ represents the event 
$\{\bY : T(\bY) \geq t_0 \}$.  Note that as the number $N$ of cyclic shifts increases, the $p$-value
$p(T(\bX))$ will converge in probability to $Q_{\blds X} ( T \geq t_0 )$

\subsection{Principal Result}
\label{sec:PR}

Our principal result requires an invariance condition on the global statistic $T$.  Informally, the condition ensures
that $T$ does not give special treatment to any column of the data matrix.

\vskip.1 in

\noindent
{\bf Definition:} A statistic $T : {\cal A}^{n \times m} \to \real$ is 
{\em invariant under constant shifts} if $T(\bX) = T(\bX')$
whenever $\bX'$ is obtained from $\bX$ by applying the {\em same} cyclic shift $\sigma_k(\cdot)$ 
to each row of $\bX$.

\vskip.1in
\noindent
The maximum column sum statistic used in the cyclic shift testing procedure is clearly
invariant under constant shifts.  More generally, any statistic of the form 
$T(\bX) = g(h(\bX_{\cdot 1}), \ldots, h(\bX_{\cdot m}))$ where $h: {\cal A}^n \to {\cal B}$
is an arbitrary local statistic (not necessarily a sum),
and $g: {\cal B}^m \to \real$ is invariant under cyclic shifts will be invariant
under constant shifts.  The 
following result establishes the asymptotic validity of the cyclic shift procedure in this
general setting.

\begin{theorem}
\label{thm1}
Let $\bX$ be a random $n \times m$ matrix whose rows $\bX_{i \cdot}$ are
independent copies of a first-order stationary ergodic Markov chain with countable 
state space ${\cal{A}}$ and transition probabilities $p(u | v)$.  Suppose that
\begin{equation}
\label{conds}
\max_{u, v \in {\cal{A}}} \ p(u | v) < 1
\ \mbox{ and } \ 
\frac{p_1(u) \, p_1(v)}{p_2(u, v)} < \infty \mbox{ for each $u,v \in {\cal{A}}$}
\end{equation}
where in the second condition we define $0 / 0$ to be $0$.  
Here $p_1(\cdot)$ and $p_2(\cdot,\cdot)$ denote the one- and two-dimensional marginal distributions
of the Markov chain, respectively.  For $m \geq 1$ let $T_m : {\cal A}^{n \times m} \to \real$
be a statistic that is invariant under constant shifts.
Then
\[
\max_{B \subseteq {\mathbb{R}}} 
\vert P_{\blds X} ( T_m \in  B ) - Q_{\blds X} ( T_m \in B ) \vert
\]
\noindent
tends to zero in probability as $m$ tends to infinity.  
\end{theorem}

The first condition in (\ref{conds}) ensures that there are not deterministic transitions between
the states of the Markov chain.  The second condition can be expressed equivalently
as $p_2(u,v) = 0$ implies $p_1(u) \, p_1(v) = 0$.
The proof of Theorem \ref{thm1} is given in Section \ref{sec4}.  
As an immediate corollary of the theorem, 
we find that
\[
\sup_{t} 
\vert P_{\blds X} ( T_m \geq t ) - Q_{\blds X} ( T_m \geq t ) \vert
\]
tends to zero in $P_m$-probability as $m$ tends to infinity.  
Thus, under the conditions of the theorem, when $m$ and $N$ are
large, the percentile based $p$-value $p(T(\bX)$ will be close to the true conditional probability
$P_{\blds X} ( T_m \geq t_0 )$ that $T_m(\bX)$ exceeds the observed value of $T_m$. 
If we define
$Q_m (A)$ to be $\E Q_{\blds X} ( A )$, where the expectation is taken under $P_m$, then 
conditional convergence also yields the unconditional result
\[
\sup_{t} 
\vert P_m( T_m \geq t ) - Q_m( T_m \geq t ) \vert \to 0
\]
as $m$ tends to infinity.  Thus, under the assumptions of Theorem \ref{thm1}, 
the percentile based $p$-value provides asymptotically correct type I error rates.

Theorem \ref{thm1} can be extended in a number of directions.  Under conditions similar to 
those in (\ref{conds}) the theorem extends to matrices $\bX$ whose rows are
independent copies of a $k$th order ergodic 
Markov chain, where $k \geq 2$ is fixed and finite.
The theorem can also be extended to settings in which
the rows of $\bX$ are independent stationary ergodic Markov chains with {\em different} transition probabilities.
In this case we require that the conditions (\ref{conds}) hold for each row-chain.

Theorem \ref{thm1} can also be extended to the setting in which the rows of $\bX$ are
independent copies of a first-order stationary ergodic Markov chain with a continuous 
state space and a transition probability density $f(u | v)$.  The existence of the transition
probability density obviates the need for the first condition in (\ref{conds}) and the 
analysis of Lemmas \ref{lem2} and \ref{lem3} in Section \ref{sec4}.  The second condition 
of (\ref{conds}) is replaced by the assumption
\begin{equation}
\label{OP1C}
\frac{f_1(X_{1}) \, f_1(X_{m})}{f_2(X_{1}, X_{m})} \ = \ O_P(1) ,
\end{equation}
where $f_1(\cdot)$ and $f_2(\cdot,\cdot)$ denote the one- and two-dimensional marginal densities
of the Markov chain, respectively.  
Markovity and ergodicity ensure that $(X_1,X_m)$ converges weakly to a pair $(X,X')$ consisting of independent copies
of $X_1$, and therefore condition (\ref{OP1C}) holds if the ratio 
$f_1(u) \, f_1(v) / f_2(u,v)$ is continuous on $\real^2$.  Thus Theorem \ref{thm1} applies, for example,
to standard Gaussian AR(1) models.  As in the discrete case, one may extend the theorem to settings in which
the rows of $\bX$ are independent stationary ergodic Markov chains with {\em different} transition probabilities,
provided that (\ref{OP1C}) holds for each row-chain.

\subsection{Illustration of Resampling Distributions}

Here we present simulation results illustrating the resampling distributions 
$P_{\blds X}(A)$ and $Q_{\blds X}(A)$ defined above.  Each simulation was conducted using an 
$n \times m$ matrix $\bX$ with independent, identically distributed
rows generated by a stationary first-order $r$-state Markov chain with a fixed transition matrix $M$.
Figure \ref{fig1} shows empirical cumulative distribution functions (cdfs) $P_{\blds X}(T < t)$ and $Q_{\blds X}(T < t)$ based on simulations conducted with $r = 5,\ n = 4$, and $m = 10$ or 50.  Each panel is based on an observed matrix $\bX$ produced by the Markov chain, and the results presented here are representative of those obtained from other simulations.  Based on Theorem \ref{thm1}, we expect the cdfs to converge as the number of columns 
$m$ increases. Accordingly, the two curves in each panel of part B of Figure \ref{fig1} 
($m$ = 50) exhibit a greater level of concordance than those in part A ($m$ = 10).  Additional simulation results based on an AR(1) model are presented in Section \ref{app}, the Appendix.

\setcounter{figure}{0}
\begin{figure}[ht]
\begin{center}
\includegraphics[scale = .5]{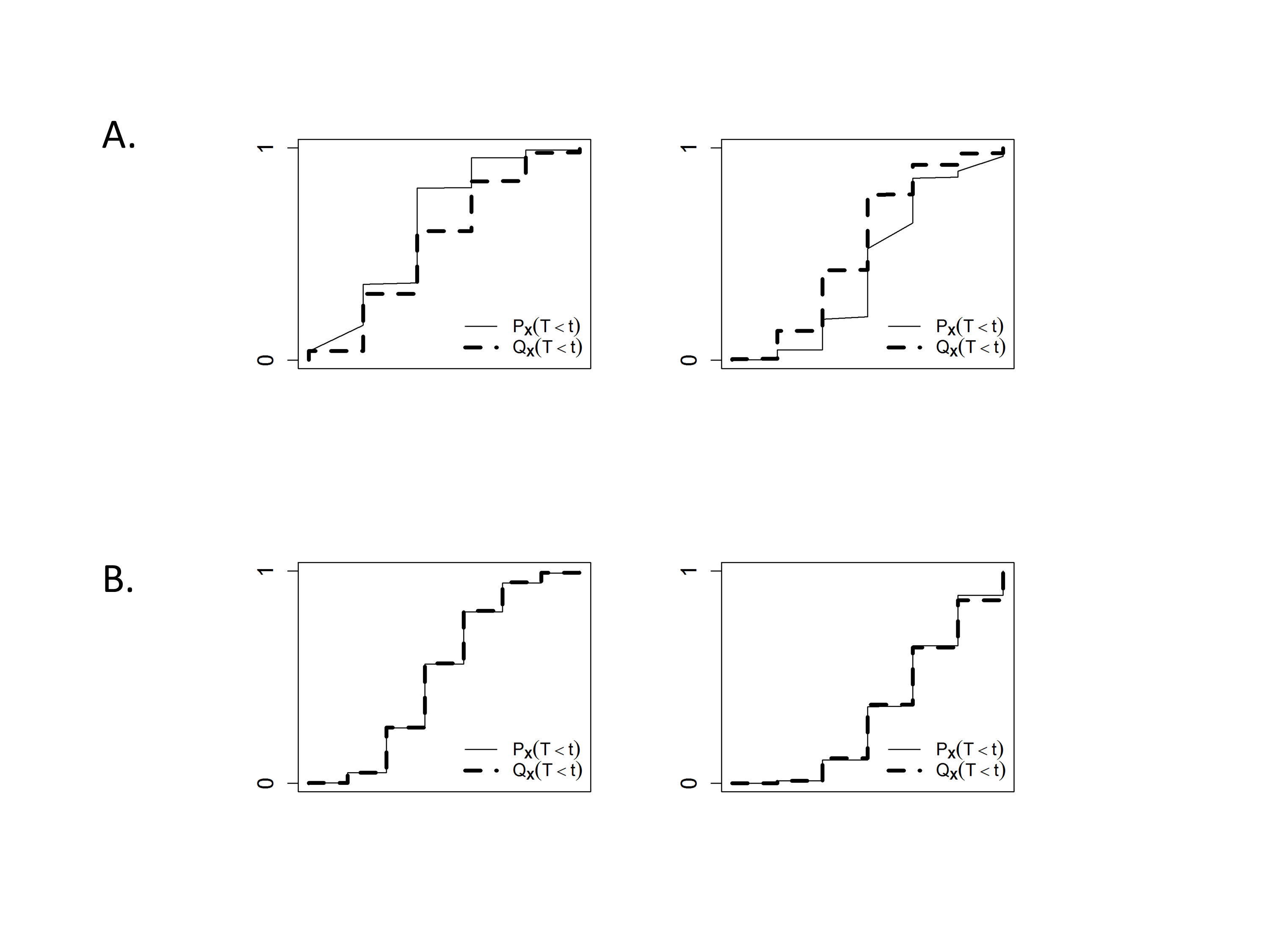}
\caption{Illustration of Resampling Distributions.  Empirical cumulative distribution functions for simulated matrices $\bX$ in which independent rows $\bX_{i \cdot}$ are generated by a first-order finite-state Markov chain.  Each panel corresponds to a simulated $4 \times m$ matrix $\bX$ with $m = 10$ (A) or $m = 50$ (B).}
\label{fig1}
\end{center}
\end{figure}

\section{Application to Genomic Data}
\label{sec3}

In tumor studies DNA copy number values for each subject are measured with respect to a normal reference, typically either a paired normal sample or a pooled reference.  In the autosomes the normal DNA copy number is two.  Underlying genomic instability in tumor tissue can result in DNA copy number gains and losses, and often these changes lead to increased or decreased expression, respectively, of affected genes (Pinkel and Albertson 2005). \nocite{pinkel2005}  Some of these genetic aberrations occur at random locations throughout the genome, and these are termed {\textit{sporadic}}.  In contrast, {\textit{recurrent}} aberrations are found in the same genomic region in multiple subjects.  It is believed that recurrent aberrations arise because they lead to changes in gene expression that provide a selective growth advantage.  Therefore regions containing recurrent aberrations are of interest because they may harbor genes associated with the tumor phenotype.  Distinguishing sporadic and recurrent aberrations is largely a statistical issue, and the cyclic shift procedure was designed to perform this task.

DNA methylation values for a given subject are also measured with respect to a paired or pooled normal reference.  Although DNA methylation values are not constant across the genome, even in normal tissue, at a fixed location they are quite stable in normal samples from a given tissue type.  Epigenetic instability can disrupt normal methylation patterns, leading to methylation gains and losses, and these changes can affect gene expression levels (Laird 2003). \nocite{laird2003}  
Regions of the genome that exhibit recurrent hyper- or hypo-methylation in tumor tissue are of interest. 

\subsection{Peeling}

In many applications more than one atypical marker may be present, 
and as a result multiple columns may produce summary statistics with extreme values.  
In tumor tissue, for example, underlying genomic instability can result in gains and losses 
of multiple chromosomal regions; likewise, epigenetic instability can lead to aberrant 
patterns of DNA methylation throughout the genome.  
In order to identify multiple atypical markers and assess their statistical significance it is necessary
to remove the effect of each discovered marker before initiating a search for the next marker.
This task is carried out by a process known as {\textit{peeling}}.  Several peeling procedures have been
proposed in the literature, including those employed by GISTIC (Beroukhim et al.\ 2007) and DiNAMIC (Walter et al.\ 2011). \nocite{beroukhim2007} \nocite{walter2011}
In the applications here we make use of the procedure described in detail in Walter et al.\ (2011) \nocite{walter2011}.  

\subsection{DNA Copy Number Data}

Walter et al.\ (2011) \nocite{walter2011} used the cyclic shift procedure to analyze the Wilms' tumor data of Natrajan et al.\ (2006). \nocite{natrajan2006}  Here we apply the procedure to the lung adenocarcinoma dataset of
 Chitale et al.\ (2009), \nocite{chitale2009} with $n$ = 192 and $m$ = 40478.  We detected a number of highly significant findings under the null hypothesis that no recurrent copy number gains or losses are present.  Table \ref{tab1} lists the genomic positions of the the three most significant copy number gains and losses, as well as neighboring genes, most of which are known oncogenes and tumor suppressors.  Strikingly, Weir et al.\ (2007) \nocite{weir2007} detected highly significant gains of the oncogenes {\it{TERT}}, {\it{ARNT}}, and {\it{MYC}} in their comprehensive investigation of the disease, each of which appears in Table \ref{tab1}.  The loss results for chromsomes 8 and 9 in Table \ref{tab1} are also highly concordant with previous findings of Weir et al.\ (2007), \nocite{weir2007} and Wistuba et al.\ (1999). \nocite{wistuba1999}  Weir et al.\ (2007) detected chromosomal loss in a broad region of 13q that contains the locus in Table \ref{tab1}, but it is not clear if the target of this region is the known tumor-suppressor {\textit{RB1}} or some other gene.

\begin{table}
\begin{center}
\caption{Genomic locations of the three most significant DNA copy number gains (top table) and losses (bottom table) found by applying the cyclic shift procedure to the lung adenocarcinoma dataset of Chitale et al.\ (2009).}

\medskip
\begin{small}
\label{tab1}
\begin{tabular}{crc} \hline
Chromosome & Gain Locus (bp) & Gene \\ \hline
5p15 & 967984 & \textit{TERT} \\
1q21 & 149346163 & \textit{ARNT} \\
8q24 & 128816933 & \textit{MYC} \\
\end{tabular}

\medskip
\begin{tabular}{crc} \hline
Chromosome & Loss Locus (bp) & Gene\\ \hline
8p23 & 2795183 & \textit{CSMD1} \\
13q11 & 19254995 & \textit{PSPC1} \\
9p21 & 21958070 & \textit{CDKN2A} \\
\end{tabular}
\end{small}
\end{center}
\end{table}

\subsection{DNA Methylation Data}

Using unsupervised clustering techniques, Fackler et al.\ (2011) \nocite{fackler2011} found an association between methylation patterns and estrogen-receptor status in a cohort of breast cancer tumors.  This cohort consisted of 20 tumor/normal pairs, and we used
differences in methylation signal between tumor and normal tissue as the observations.  We applied the cyclic shift procedure to the resulting differences to detect loci that exhibited recurrent hyper- or hypomethylation in tumors.  As shown in Table \ref{tab2}, the most significant hypermethylation sites occur in {\it{ABCA3}}, {\it{GALR1}}, and {\it{NID2}}, and these genes have previously been found to be highly methylated in lung adenocarcinoma, head and neck squamous cell carcioma, and bladder cancer, respectively,
by Selamat et al.\ (2012), \nocite{selamat2012} Misawa et al.\ (2008), \nocite{misawa2008} and Renard et al.\ (2009). \nocite{renard2009}  Hypomethylation of the transcription factor {\it{MYT1}} on chromosome 20 was detected; this is notable because Vir\'{e} et al.\ (2006) 
found that {\it MYT1} could be activated via decreased methylation.

\begin{table}
\begin{center}
\caption{Genomic locations of the three most significant hypermethylation (top table) and hypomethylation (bottom table) sites found by applying the cyclic shift procedure to the breast cancer dataset of Fackler et al.\ (2011).}

\medskip
\begin{small}
\label{tab2}
\begin{tabular}{crc} \hline
Chromosome & Gain Locus (bp) & Gene \\ \hline
16p13 & 2331829 & \textit{ABCA3} \\
18q23 & 73091357 & \textit{GALR1} \\
14q22 & 51605897 & \textit{NID2} \\
\end{tabular}

\medskip
\begin{tabular}{crc} \hline
Chromosome & Gain Locus (bp) & Gene \\ \hline
20q13 & 62266251 & \textit{MYT1} \\
3q24 & 144378065 &  \textit{SLC9A9} \\
1q21 & 150565702 & \textit{MCL1} \\
\end{tabular}
\end{small}
\end{center}
\end{table}

\subsection{Meta-Analysis of Genomewide Association Studies}

Genome-wide association studies (GWAS) are used to identify genetic markers, typically {\textit{single nucleotide polymorphisms}} (SNPs), that are associated with a disease of interest.  When conducting a GWAS involving a common disease and alleles with small to moderate effect sizes, large numbers of cases and controls are required to have adequate power to detect disease SNPs (Pfeiffer et al.\ 2009). \nocite{pfeiffer2009}  

The Welcome Trust Case Control Consortium (WTCCC 2007) performed a genome-wide association study of seven common familial diseases - bipolar disorder (BD), coronary artery disease (CAD), Crohn's disease (CD), hypertension (HT), rheumatoid arthritis (RA), type I diabetes (T1D), and type II diabetes (T2D) - based on an analysis of 2000 separate cases for each disease and a  set of 3000 controls.  
We applied the inverse of the standard normal cumulative distribution function to the Cochran-Armitrage trend test $p$-values from the WTCCC study, a transformation that produces z-scores whose values are similar those exhibited by a stationary process.  We then analyzed the matrix $\bX$ whose entries are negative thresholded z-scores arranged in rows corresponding to the seven disease phenotypes.  As seen in Figure \ref{fig2}, a number of regional markers on chromosome 6 produce extremely large column sums.  These markers lie in the major histocompatability complex (MHC), which is noteworthy because MHC class II genes have been shown to be associated with autoimmune disorders, including RA and T1D (Fernando et al.\ 2008).\nocite{fernando2008}  When applied to $\bX$, cyclic shift testing identified several highly significant apparently pleiotropic SNPs in the MHC region that produced large entries in the rows corresponding to both RA and T1D, including rs9270986, which is upstream of the RA and T1D susceptibility gene {\textit{HLA-DRB1}}.  

The WTCCC  dataset serves as a proof of principle for cyclic shift applied to GWAS studies, although the use of a common set of controls may create modest additional correlation not fully captured in the cyclic shifts. We note that the cyclic shift procedure appplied to GWAS is sensitive only to small $p$-values that occur in multiple studies.  Thus the procedure is qualitatively different from typical meta-analyses, such as Zeggini et al.\ (2008), \nocite{zeggini2008} which
can be sensitive to large observed effects form a single study.

\setcounter{figure}{1}
\begin{figure}[ht]
\begin{center}
\includegraphics[scale = .5]{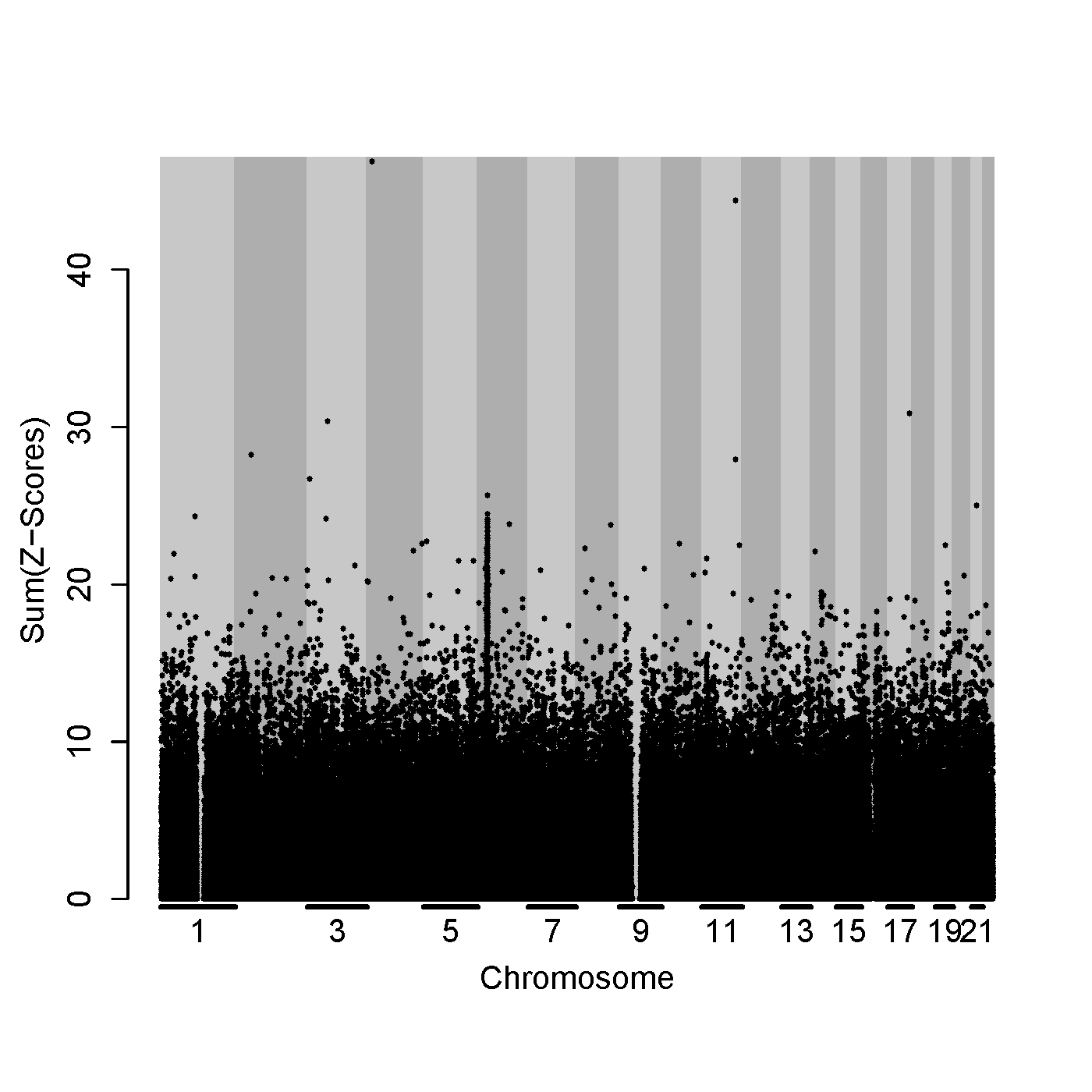}
\caption{Cyclic Shift Testing Identifies Pleiotropic Single Nucleotide Polymorphisms.  Marker-specific summary statistics were obtained from the Welcome Trust Case Control Consortium study and plotted genome-wide.  Numerous regional markers in the multihistocompatability complex region of chromosome 6 exhibit large summary statistics, including several markers that were highly significant under cyclic shift testing and were associated with multiple disease phenotypes.}
\label{fig2}
\end{center}
\end{figure}

\section{Proof of Theorem \ref{thm1}}
\label{sec4}

Let $\bX$ be a random $n \times m$ matrix whose rows
$\bX_{1 \cdot}, \dots, \bX_{n \cdot}$ are independent
realizations of a first-order stationary ergodic Markov chain with countable state space ${\cal{A}}$.
Denote the distribution of $\bX$ in ${\cal A}^{n \times m}$ by $P_m$. 
Let $p_1( \cdot )$ and $p( \cdot \vert \cdot )$ denote, respectively, the stationary distribution and the
one-step transition probability of the Markov chain defining the rows of $\bX$.
Let $p_l( \cdot )$ denote the joint probability mass function
of $l$ contiguous variables in the chain.
Thus the vectors $\bX_{i \cdot}$ have common probability mass
function
\begin{equation}
\label{MC}
p_m(x_0, \dots, x_{m - 1}) \ = \ p_1(x_0) \prod_{j = 1}^{m - 1} p(x_j \vert x_{j - 1}) .
\end{equation}
In what follows we assume that (\ref{conds}) holds.
The ergodicity assumption on the Markov chain ensures 
that the joint probability mass function of $(X_{0}, X_{m-1})$ converges 
to the joint probability mass function of the pair $(X,X')$ where $X, X' \in {\cal A}$ are independent with the
same distribution as $X_{0}$.  It follows that
\begin{equation}
\label{OP1}
\frac{p_1(X_{i,0}) \, p_1(X_{i,m-1})}{p_2(X_{i,0}, X_{i,m-1})} \ = \ O_P(1) \ \ 1 \leq i \leq n .
\end{equation}
In other words, for each row $i$ the ratio in (\ref{OP1}) is stochastically bounded under 
$P_m$ as $m$ tends to infinity.  

Suppose for the moment that $m$ is fixed.  For any integer $r$, define the cyclic 
shift $\sigma_r : {\cal A}^m \to {\cal A}^m$ on sequences
of length $m$ by
\[
\sigma_{r}(x_0, x_1, \dots, x_{m - 1}) = (x_{[r]}, x_{[r + 1]}, \dots, x_{[r + (m - 1)]})
\]
where $[k] = k \mbox{ mod } m$.  We index vectors as $(x_0, x_1, \dots, x_{m - 1})$
rather than $(x_1, x_2, \dots, x_m)$,
as was done in the body of the present manuscript, because
this allows us to write the subscripts of the shifted vector in terms of $[k]$,
substantially reducing notation.
For each $\br = (r_1, \dots, r_n) \in {\mathbb{Z}}^n$ define
$\sigma_{\blds r}(\bX)$ to be the $n \times m$ matrix
with rows $\sigma_{r_1}(\bX_{1 \cdot}), \dots, \sigma_{r_n}(\bX_{n \cdot})$.
If ${\br}, {\bs} \in {\mathbb{Z}}^n$, then it is easy to verify that
\[
( \sigma_{\blds r} \circ \sigma_{\blds s}) (\bX)
\ = \
(\sigma_{\blds s} \circ \sigma_{\blds r}) (\bX)
\ = \
(\sigma_{{\blds r} + {\blds s}}) (\bX) .
\]
Let ${\cal S}_m (\bX) = \{ \sigma_{\blds r} (\bX) : \br \in \mathbb{Z}^n \}$ be the set of cyclic shifts of $\bX$.

Let $P_{\blds X}$ and $Q_{\blds X}$ be the true conditional and cyclic conditional distributions given ${\cal S}_m (\bX)$,
defined by $P_{\blds X} (A) = P_m( A \, | \, {\cal S}_m(\bX))$ and 
$Q_{\blds X} ( A ) = m^{-n} \, |A \cap {\cal S}_m(\bX)|$, respectively.
In order to compare the distributions $P_{\blds X}$ and $Q_{\blds X}$ we introduce two closely 
related distributions, $P_{\blds X}^o$ and
$Q_{\blds X}^o$, that are more amenable to analysis.  Let $P_{\blds X}^o$ be a (random) measure on
${\cal{A}}^{n \times m}$ defined by
\[
P_{\blds X}^o( A )
\ = \
\sum_{{\blds r} \in [m]^n} \eta(\sigma_{\blds r}(\bX)) \, I ( \sigma_{\blds r}(\bX) \in A ),
\]
where $[m] = \{0, 1, \ldots, m-1\}$ and
\[
\eta(\sigma_{\blds r}(\bX))
\ = \
\prod_{i=1}^n
\left[
\frac{ p_m(\sigma_{r_i}(\bX_{i \cdot})) }{ \sum_{s \in [m]} p_m(\sigma_s(\bX_{i \cdot})) }
\right].
\]
Let $Q_{\blds X}^o$ be a (random) measure on ${\cal{A}}^{n \times m}$ defined by
\[
Q_{\blds X}^o (A)
\ = \
\sum_{{\blds r} \in [m]^n} \frac{1}{m^n} \, I ( \sigma_{\blds r} (\bX) \in A ) .
\]
One may readily verify that
$P_{\blds X}^o({\cal{A}}^{m \times n}) =
Q_{\blds X}^o({\cal{A}}^{m \times n}) = 1$,
so both $P_{\blds X}^o$ and $Q_{\blds X}^o$ are valid probability measures on ${\cal{A}}^{m \times n}$.

We will say that the set of cyclic shifts ${\cal{S}}_m(\bX)$ is {\it full} if its cardinality
is equal to $m^n$, or equivalently, if all cyclic shifts of $\bX$ are distinct.

\begin{lemma}
If ${\cal{S}}_m(\bX)$ is full, then
(a) $P_{\blds X}^o = P_{\blds X}$
and (b) $Q_{\blds X}^o = Q_{\blds X}$.
\end{lemma}

\begin{proof}
(a)  For any $A \subseteq {\cal{A}}^{n \times m}$ we may write $P_{\blds X}(A)$ as
\begin{equation}
\label{cond1}
\frac{P_m(A \cap {\cal{S}}_m(\blds X))}{P_m({\cal{S}}_m(\bX))} = \sum_{{\blds r} \in [m]^n} \frac{P_m(\sigma_{\blds r}(\bX))}{P_m({\cal{S}}_m(\bX))} I(\sigma_{\blds r}(\bX) \in A).
\end{equation}
\noindent
Since ${\cal{S}}_m(\bX)$ is full, 
\[
P_m({\cal{S}}_m(\bX)) = P_m \left( \bigcup_{{\blds r} \in [m]^n} \sigma_{\blds r}(\bX) \right) = \sum_{{\blds r} \in [m]^n} P_m(\sigma_{\blds r}(\bX)).  
\]
The independence of the rows of $\bX$ allows us to write the last expression as $\displaystyle{\sum_{\blds r \in [m]^n} \prod_{i = 1}^n p_m(\sigma_{r_i}(\bX_{i \cdot}))}$, but this may be rewritten as $\displaystyle{\prod_{i = 1}^n \sum_{s \in [m]} p_m(\sigma_{s}(\bX_{i \cdot}))}$.  Therefore (\ref{cond1}) is equivalent to
\[
\sum_{{\blds r} \in [m]^n} \left[ \prod_{i = 1}^n \frac{p_m(\sigma_{r_i}(\bX_{i \cdot}))}{\sum_{s \in [m]} 
p_m(\sigma_s(\bX_{i \cdot}))} \right] I(\sigma_{\blds r}(\bX) \in A) = P_{\blds X}^o(A).
\]

\noindent
(b) There are $m^n$ elements in ${\cal{S}}_m (\bX)$ when $\bX$ is full, so for any $A \in \cal{A}$
\[
Q_{\blds X} ( A ) = m^{-n} \, |A \cap {\cal S}_m(\bX)| = \sum_{{\blds r} \in [m]^n} \frac{1}{m^n} \, I ( \sigma_{\blds r} (\bX) \in A ) = Q_{\blds X}^o (A).
\]
\end{proof}

\begin{lemma}
\label{lem2}
Let $\mathrm{\mathbf x} = (x_0, x_1, \dots, x_{m - 1}) \in {\cal{A}}^m$ be a sequence of length $m$.
Let $k$ be the least positive integer such that $\sigma_k(\mathrm{\mathbf x}) = \mathrm{\mathbf x}$.
If $k < m$, then $k$ divides $m$, and $\mathrm{\mathbf x}$ is equal to the repeated concatenation
of a fixed block of length $k$.
\end{lemma}

\begin{proof}
Suppose to the contrary that $1 \leq k < m$ does not divide $m$.  Then we may write $m = kq + r$,
where $1 \leq r < k$.  Now $\sigma_k({\bx}) = {\bx}$ implies that
$\sigma_{-k}({\bx}) = {\bx}$, and it follows that
\[
\sigma_r({\bx}) = \sigma_{m - kq}({\bx}) = \sigma_m \circ \sigma_{-kq}({\bx}) = {\bx} .
\]
As this contradicts the minimality of $k$, we conclude that $k$ divides $m$.   The second conclusion
follows in a straightforward way from the first.
\end{proof}

\setcounter{corollary}{0}
\begin{corollary}
If $\mathrm{\mathbf x} \in {\cal{A}}^m$ is such that $\sigma_k(\mathrm{\mathbf x}) = \mathrm{\mathbf x}$ for some $1 \leq k < m$, then $\mathrm{\mathbf x}$ contains two disjoint,
equal blocks of length at least $m / 3$.
\end{corollary}

\begin{lemma}
\label{lem3}
If (\ref{conds}) holds, then $P_m({\cal{S}}_m(\mathrm{\mathbf X})\ \text{is full})$ converges to 1 as $m$ tends to infinity.
\end{lemma}

\begin{proof}
We begin by noting that ${\cal{S}}_m(\bX)$ is full if
${\cal{S}}_m(\bX_{i \cdot})$ is full for $i = 1, \ldots, n$.
Because the rows of $\bX$ are independent, it
therefore suffices to prove the result
in the case $n = 1$.  Thus we write $\bX = (X_0, \dots, X_{m - 1})$.  If
${\cal{S}}_m(\bX)$ is not full, then Corollary 1
implies that there exist integers $l, r \geq m / 3$ such that
$X_j = X_{r+j}$ for $j = 0,\ldots, l-1$.  An easy calculation using the Markov
property shows that, for fixed $r$, the $P_m$-probability of this event
is at most $\rho^{l-1}$, where
$\rho < 1$ is the maximum appearing in (\ref{conds}).  Thus the probability that
${\cal{S}}_m(\bX)$ is not full is at most $m \, \rho^{m / 3 - 1}$,
which tends to zero as $m$ tends to infinity.
\end{proof}

\vskip.1in

\noindent
{\bf Definition:}
A set $A \subset {\cal{A}}^{n \times m}$ is {\it invariant under constant shifts} if $\sigma_{\blds r}(A) = A$ whenever
$\br = (r,\ldots,r)$ is a constant index sequence.  Let $\mathbb{A}_m$ be the family of all sets 
$A \subset {\cal{A}}^{n \times m}$ that are invariant under constant shifts. 

\vskip.1in

\begin{theorem}
\label{thm2} 
Suppose that (\ref{conds}) holds and that the
stationary Markov chain
described by (\ref{MC}) is ergodic.
Then
\[
\max_{A \in \mathbb{A}_m} \vert P_{\blds X}^o( A ) - Q_{\blds X}^o( A ) \vert
\ \to \ 0
\]
\noindent
in probability as $m$ tends to infinity.
\end{theorem}

\begin{proof}
Fix $m \geq 1$ and $A \in \mathbb{A}_m$.  For $k \in {\mathbb{Z}}$ let
${\bk}^{\ast} = (k, k, \dots, k) \in {\mathbb{Z}}^n$ be the constant
sequence each of whose coordinates is equal to $k$.  It follows from the
invariance of $A$ and the basic properties of cyclic shifts that for each $k \in {\mathbb{Z}}$,
\begin{eqnarray*}
P_{\blds X}^o (A)
& = &
\sum_{{\blds r} \in [m]^n} \eta(\sigma_{\blds r}(\bX)) \, I ( \sigma_{\blds r}(\bX) \in A ) \\
& = &
\sum_{{\blds r} \in [m]^n} \eta(\sigma_{\blds r + \blds k^*}(\bX)) \, I ( \sigma_{\blds r + \blds k^*}(\bX) \in A ) \\
& = &
\sum_{{\blds r} \in [m]^n} \eta(\sigma_{\blds r + \blds k^*}(\bX)) \, I ( \sigma_{\blds r}(\bX) \in A ) .
\end{eqnarray*}
Thus we may express $P_{\blds X}^o ( A )$ in the form of an average over $k$:
\[
P_{\blds X}^o ( A )
\ = \
\sum_{{\blds r} \in [m]^n}
\left[ \frac{1}{m} \sum_{k \in [m]} \eta(\sigma_{\blds r + \blds k^*}(\bX)) \right] I ( \sigma_{\blds r}(\bX) \in A ) .
\]
Combining this last expression with the definition of $Q_{\blds X}^o$ yields the bound
\begin{eqnarray}
\label{prdiff}
| P_{\blds X}^o ( A ) - Q_{\blds X}^o ( A ) |
& \leq &
\sum_{{\blds r} \in [m]^n}
\left| \frac{1}{m} \sum_{k \in [m]} \eta(\sigma_{\blds r + \blds k^*}(\bX)) -
\frac{1}{m^n} \, \right| I ( \sigma_{\blds r}(\bX) \in A )  \nonumber \\
& \leq &
\sum_{{\blds r} \in [m]^n}
\left| \frac{1}{m} \sum_{k \in [m]} \eta(\sigma_{\blds r + \blds k^*}(\bX)) - \frac{1}{m^n} \, \right| \nonumber \\
& = &
\frac{1}{m^n}  \sum_{{\blds r} \in [m]^n}
\left| \frac{1}{m} \sum_{k \in [m]} m^n \eta(\sigma_{\blds r + \blds k^*}(\bX)) - 1 \, \right| .
\end{eqnarray}

We now turn our attention to the quantity
$m^n \, \eta(\sigma_{\blds r + \blds k^*}(\bX))$
appearing in ({\ref{prdiff}).
Let $\bx = (x_0, \dots, x_{m - 1})$ be a fixed $m$-vector with
entries in ${\cal{A}}$, and let $t \in {\mathbb{Z}}$.  By expanding the
joint probability $p_m( \cdot )$ as a product of one-step conditional
probabilities and canceling common terms, a straightforward calculation
shows that for all integers $t$
\begin{equation}
\label{rhodelt}
m \cdot \frac{ p_m(\sigma_{t}(\bx)) }{ \sum_{s \in [m]} p_m(\sigma_s(\bx)) }
\ = \
\rho_{t}(\bx) \, \gamma_m^{-1}(\bx)
\end{equation}
where
\[
\rho_t(\bx) = \frac{ p_1(x_{[t]}) \, p_1(x_{[t-1]}) }{ p_2(x_{[t]},x_{[t-1]}) }
\ \mbox{ and } \
\gamma_m(\bx)
\, = \,
\frac{1}{m} \sum_{j=0}^{m-1} \rho_j(\bx) .
\]
(Recall that $[t] = t \ {\text{mod}}\ m$.)
It follows from the definition of $\eta(\bX)$ and equation (\ref{rhodelt}) that
\begin{eqnarray*}
m^n \, \eta(\sigma_{\blds r + \blds k^*}(\bX))
& = &
\prod_{i=1}^n
\left[
\frac{ m \, p_m(\sigma_{r_i+k}(\bX_{i \cdot})) }
{ \sum_{s \in [m]} p_m(\sigma_s(\bX_{i \cdot})) }
\right] \\
& = &
\prod_{i=1}^n \rho_{r_i + k}(\bX_{i \cdot}) \, \gamma_m^{-1}(\bX_{i \cdot})
\ = \
\Gamma_m^{-1}(\bX) \, \prod_{i=1}^n \rho_{r_i + k}(\bX_{i \cdot})
\end{eqnarray*}
where $\Gamma_m(\bX) = \prod_{i=1}^n \gamma_m(\bX_{i \cdot})$.

The assumptions of the theorem ensure that the random variables
$X_{i,0}$, \ldots, $X_{i,m-1}$ in the $i$th row of $\bX$ are the initial terms
of a stationary ergodic process, and therefore the same is true of the non-negative
random variables $\rho_1(\bX_{i \cdot})$,\ldots,$\rho_{m-1}(\bX_{i \cdot})$.
Note that the random variable $\rho_0(\bX_{i \cdot})$ cannot be included in this sequence
because it involves the non-adjacent variables $X_{i,0}$ and $X_{i,m-1}$.  
It is easy to see that
\[
\E \rho_1(\bX_{i \cdot})
\ = \
\E \left( \frac{ p_1(X_{i,1}) \, p_1(X_{i,0}) }{ p_2(X_{i,0}, X_{i,1} ) }  \right)
\ = \
\sum_{u,v \in {\cal{A}}} \frac{ p(u) \, p(v) }{ p(u,v) } \, p(u,v)
\ = \
1 .
\]
From the ergodic theorem and the fact that $\rho_0(\bX_{i \cdot})$ is stochastically bounded
(see (\ref{OP1})), it follows that
\begin{equation}
\label{gmop}
\gamma_m({\bf X}_{i \cdot}) \ = \  \E \rho_1(\bX_{i \cdot})  + o_P(1) \ = \ 1 + o_P(1) ,
\end{equation}
and therefore
$\Gamma_m(\bX)$ and $\Gamma_m^{-1}(\bX)$ are equal to $1 + o_P(1)$ as well.
(Here and in what follows the stochastic order symbols $o_P(1)$ and $O_P(1)$
refer to the underlying measure $P_m$ with $m$ tending to infinity).
For $\br \in [m]^n$ let 
$V_0(\br) = \{ k \in [m] : r_i + k \equiv 0 \mbox{ mod $m$ for some } 1 \leq i \leq n \}$
and let $V_1(\br) = [m] \setminus V_0(\br)$.  Note that $|V_0(\br)| \leq n$ for each 
$\br \in [m]^n$.
Combining the relation (\ref{gmop}) with inequality (\ref{prdiff}) and
equation (\ref{rhodelt}), we conclude that
\begin{eqnarray}
\lefteqn{| P_{X}^o ( A ) - Q_{X}^o ( A ) |} \nonumber \\[.1in]
& \leq &
\Gamma_m^{-1}(\bX) \cdot
\frac{1}{m^n}  \sum_{{\blds r} \in [m]^n}
\left| \frac{1}{m} \sum_{k \in [m]} \, \prod_{i=1}^n \rho_{r_i + k}(\bX_{i \cdot})  - 1 \, \right|
\ + \
| \Gamma_m^{-1}(\bX) - 1 | \nonumber \\[.1in]
& = &
O_P(1) \cdot
\frac{1}{m^n}  \sum_{{\blds r} \in [m]^n}
\left| \frac{1}{m} \sum_{k \in [m]} \, \prod_{i=1}^n \rho_{r_i + k}(\bX_{i \cdot})  - 1 \, \right|
\ + \
o_P(1) \nonumber \\[.1in]
\label{pdbnd}
& = &
O_P(1) \cdot
\frac{1}{m^n}  \sum_{{\blds r} \in [m]^n}
\left| \frac{1}{m} \sum_{k \in V_1({\blds r})} \, \prod_{i=1}^n \rho_{r_i + k}(\bX_{i \cdot})  - 1 \, \right|
\ + \
O_P(1) \, \Delta_m
\ + \ 
o_P(1) 
\end{eqnarray}
where in the last line
\[
\Delta_m
\ := \ 
\frac{1}{m^n}  \sum_{{\blds r} \in [m]^n}
\frac{1}{m} \sum_{k \in V_0({\blds r})} \, \prod_{i=1}^n \rho_{r_i + k}(\bX_{i \cdot}) .
\]
As the upper bound in (\ref{pdbnd}) is independent of our choice of $A \in \mathbb{A}_m$, it is enough to show 
that the first two terms in (\ref{pdbnd}) are $o_P(1)$.   Concerning the first term, 
by Markov's inequality it suffices to show that
\[
\max_{\blds r \in [m]^n}
\E \left| \frac{1}{m} \sum_{k \in V_1({\blds r})} \prod_{i=1}^n \rho_{r_i + k'}(\bX_{i \cdot})  - 1 \, \right|
\ \to \
0
\ \mbox{ as } \ m \to \infty .
\]
This follows from Corollary \ref{nzcor} below.  As for the second term, note that
\begin{eqnarray*}
\Delta_m
& \leq &
\prod_{i=1}^n ( \rho_{0}(\bX_{i \cdot}) \vee 1 ) \cdot
\left[ \frac{1}{m^n}  \sum_{{\blds r} \in [m]^n}
\frac{1}{m} \sum_{k \in V_0({\blds r})} \, \prod_{i: r_i + k \neq 0} 
\rho_{r_i + k}(\bX_{i \cdot}) \right] \\[.1in]
& = & 
O_P(1) \cdot
\left[ \frac{1}{m^n}  \sum_{{\blds r} \in [m]^n}
\frac{1}{m} \sum_{k \in V_0(\blds r)} \, \prod_{i: r_i + k \neq 0} 
\rho_{r_i + k}(\bX_{i \cdot}) \right] 
\end{eqnarray*}
The term in brackets is non-negative and has expectation at most $n / m$.  Thus
$\Delta_m = o_P(1)$ and the result follows.}
\end{proof}

Let $\{ U_{1}(k) : k \geq 0 \}$, \ldots, $\{ U_{n}(k) : k \geq 0 \}$ be independent,
real-valued
stationary ergodic processes defined on the same underlying probability
space.
Suppose that $\E |U_{i}(0)|$ is bounded for $i = 1,\ldots,n$, and define
$\mu = \Pi_{i=1}^n \E (U_{i}(0))$.
Let ${\bf r} = (r_1,\ldots,r_n)$ denote a vector with non-negative integer-valued
components.  For $k, m \geq 1$ define random variables
\[
V_{m} (k: {\bf r}) \ = \ \prod_{i=1}^n U_{i}((r_i + k) \, \mbox{mod} \, m) - \mu .
\]
The independence of the processes $\{U_i(\cdot)\}$ ensures that $\E( V_{m} (k : {\bf r}) ) = 0$.

\begin{lemma}
\label{shift-lln}
Under the assumptions above, $\max_{{\bf r} \in {\mathbb Z}^n} \,
\E \left| m^{-1} \sum_{k = 0}^{m-1} V_{m}(k : {\bf r}) \right|$
converges to zero as $m$ tends to infinity.
\end{lemma}

\begin{proof}
Standard arguments show that the joint process
$\{ (U_1(k), \ldots, U_n(k)) : k \geq 0 \}$ is stationary and ergodic,
and therefore the same is true for the process $\{ \prod_{i=1}^n U_{i}(k) : k \geq 0 \}$
of products.
The $L_1$ ergodic theorem implies that
\begin{equation}
\label{deltinq1}
\Delta(l)
\ = \
\E \left| \, \frac{1}{l} \sum_{k = 0}^{l-1} \left( \prod_{i=1}^n U_{i}(k) \, - \, \mu \right) \right|
\ \to \ 0 \ \mbox{ as } \ l \to \infty.
\end{equation}
Note also that
\begin{equation}
\label{deltinq2}
\Delta(l)
\ \leq \
\E \left| \, \prod_{i=1}^n U_{i}(0) \, - \, \mu \, \right|
\ \leq \ 2 \prod_{i=1}^n \E |U_{i}(0)|
\ \deq \
\Delta_0
\end{equation}
which is bounded by assumption.

Fix $m \geq 1$ and ${\bf r} = (r_1, \ldots, r_n)$ with $r_i \geq 0$.
Because the indices of $U_i(\cdot)$ in $V_m(k: {\bf r})$ are
assessed modulo $m$,
we may assume without loss of generality that
$0 \leq r_1, \ldots, r_n \leq m-1$.
Let $0 \leq r(1) < r(2) < \cdots < r(n') \leq m-1$ be the distinct
order statistics of $r_1,\ldots,r_n$, and note that $n' \leq n$.
Define $r(0) = 0$, $r(n' + 1) = m$, and the differences
$m_j = r(j+1) - r(j)$ for $j = 0,\ldots, n'$.
Consider the decomposition
\begin{equation}
\label{decomp}
\sum_{k = 0}^{m-1} V_{m}(k: {\bf r})
\ = \
\sum_{j=0}^{n'} W_j
\ \mbox{ where } \
W_j
\ =
\sum_{k = r(j)}^{r(j+1)-1} V_{m}(k: {\bf r}) .
\end{equation}
The key feature of the sum $W_j$ is this: for $r(j) \leq k \leq r(j+1) - 1$
there are no ``breaks'' in
the indexing of the terms $U_{i}((r_i + k) \, \mbox{mod} \, m)$ in $V_{m}(k: {\bf r})$
arising from the modular sum.  In particular, there exist integers $\tilde{r}_1,\ldots, \tilde{r}_n$
such that $(r_i + k) \, \mbox{mod} \, m = \tilde{r}_i + k$ for each $i = 1,\ldots, n$,
and each $k$ in the sum defining $W_j$.
As a result, the stationarity and independence of the individual
processes $\{U_i(\cdot)\}$ ensures
that $W_j$ is equal in distribution to the random variable
\[
\tilde{W}_j
\ = \
\sum_{k = 0}^{m_j-1} \left( \prod_{i=1}^n U_{i}(k) \, - \, \mu \right) .
\]

We now turn our attention to the expectation in the statement of the
lemma.
It follows immediately from the decomposition (\ref{decomp}) that
\[
\frac{1}{m} \sum_{k = 0}^{m-1} V_{m}(k: {\bf r})
\ = \
\sum_{j=0}^{n'} \frac{m_j}{m} \frac{1}{m_j} W_j ,
\]
which yields the elementary bound
\[
\left| \frac{1}{m} \sum_{k = 0}^{m-1} V_{m}(k: {\bf r}) \right|
\ \leq \
\sum_{j=0}^{n'} \frac{m_j}{m} \left| \frac{1}{m_j} W_j \right| .
\]
Taking expectations of both sides in the last display yields the inequality
\[
\E \left| \frac{1}{m} \sum_{k = 0}^{m-1} V_{m}(k: {\bf r}) \right|
\ \leq \
\sum_{j=0}^{n'} \frac{m_j}{m} \, \E \left| \frac{1}{m_j} W_j \right|
\ = \
\sum_{j=0}^{n'} \frac{m_j}{m} \, \E \left| \frac{1}{m_j} \tilde{W}_j \right|
\ = \
\sum_{j=0}^{n'} \frac{m_j}{m} \, \Delta(m_j),
\]
where the first equality follows from the distributional equivalence of
$W_j$ and $\tilde{W}_j$.
In particular, for each integer $l \geq 1$ we have
\[
\E \left| \frac{1}{m} \sum_{k = 0}^{m-1} V_{m}(k: {\bf r}) \right|
\ \leq \
\sum_{j : m_j \leq l} \frac{m_j}{m} \, \Delta(m_j) \, + \,
\sum_{j : m_j > l} \frac{m_j}{m} \, \Delta(m_j)
\ \leq \
\frac{n l \Delta_0}{m} \, + \, \sup_{l' > l} \Delta(l') .
\]
It follows from (\ref{deltinq1}) and (\ref{deltinq2}) that the final term
in the last display tends to zero with $m$ if $l = l(m)$ is
any sequence such
that $l$ tends to infinity and $1 / m$ converges to 0.  Moreover, the final term does not depend on
the vector $r$.  This completes the proof of the lemma.
\end{proof}

\vskip.2in
An elementary argument using Lemma \ref{shift-lln} establishes the following corollary.

\begin{corollary}
\label{nzcor}
Under the assumptions of Lemma \ref{shift-lln}, 
$\max_{r \in {\mathbb N}^n} \E \left| m^{-1} \sum_{k'} V_{m}(k:r) \right|$
converges to zero as $m$ tends to infinity,
where for each $r$ the sum is restricted to those $k' \in [m]$ such
that $r_i + k' \not\equiv 0 \mbox{ mod } m$.
\end{corollary}

\noindent
{\bf Proof of Theorem \ref{thm1}:} Theorem \ref{thm1} follows from Theorem \ref{thm2} and the fact
that for each $B \subseteq \real$ the event $\{ \bY : T_m(\bY) \in B \} \in \mathbb{A}_m$
as $T_m$ is invariant under constant shifts.

\section{Discussion}

High resolution genomic data is routinely used by biomedical investigators to search for recurrent
genomic aberrations that are associated with disease.  
Cyclic shift testing provides a simple, permutation based approach to identify aberrant markers
in a variety of settings.
Here we establish finite sample, large marker asymptotics for the consistency of $p$-values produced 
by cyclic shift testing.  The results apply to a broad family of Markov based null distributions.
To our knowledge, this is the first theoretical justification of a testing 
procedure of this kind.  Although cyclic shift testing was developed for DNA copy number analysis, 
we demonstrate its utility for DNA methylation and meta-analysis of genome wide association studies.

\section{Acknowledgements}

This research was supported by the National Institutes of Health (T32 CA106209 for VW), the Environmental Protection Agency (RD835166 for FAW), the National Institutes of Health/National Institutes of Mental Health (1R01MH090936-01 for FAW), and the National Science Foundation (DMS-0907177 and DMS-1310002 for ABN).

\newpage
\section{Appendix}
\label{app}
\setcounter{figure}{0} \renewcommand{\thefigure}{A.\arabic{figure}} 
\begin{figure}[ht]
\begin{center}
\includegraphics[scale = .6]{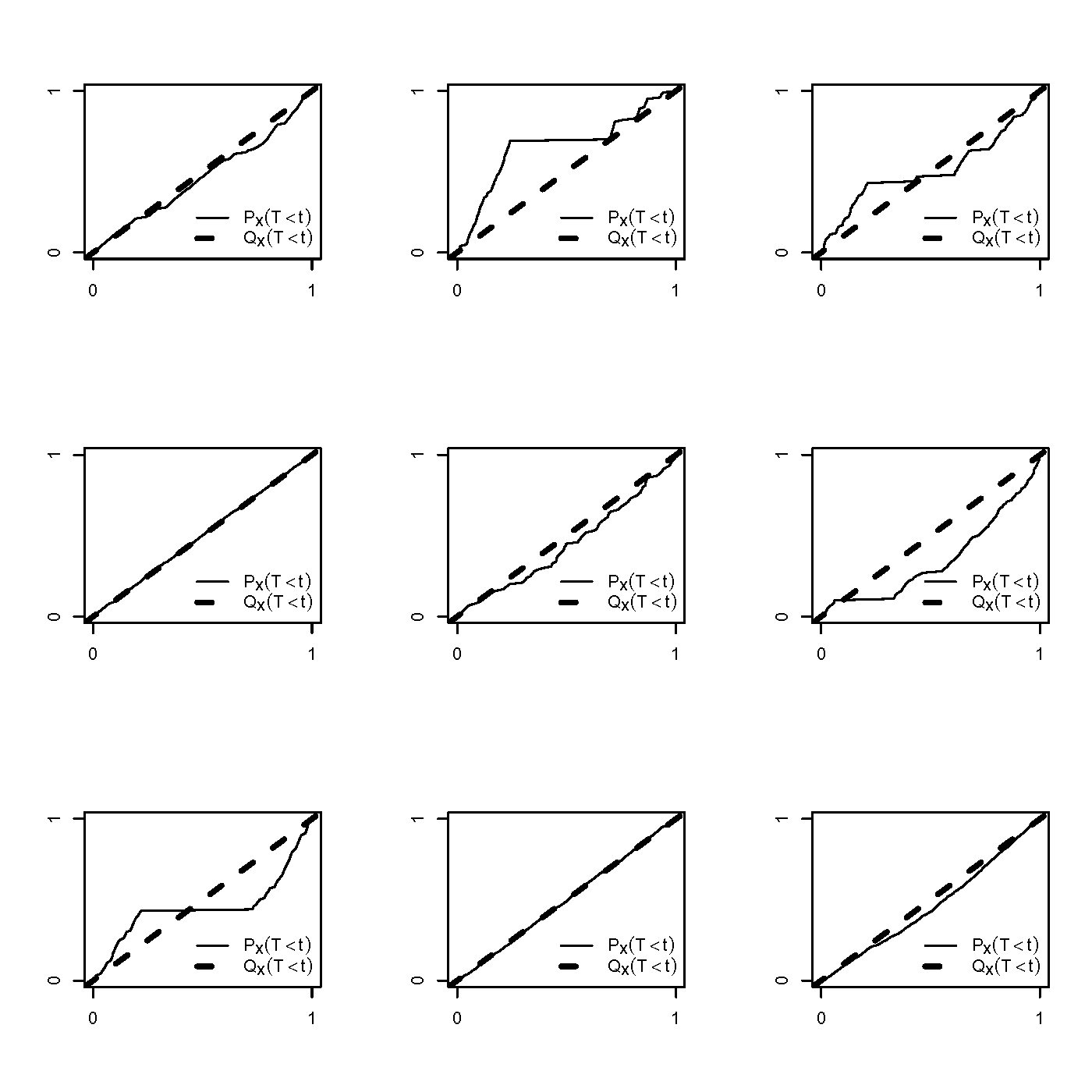}
\caption{Illustration of Resampling Distributions.  Empirical cumulative distribution functions for simulated matrices $\bX$ in which independent rows $\bX_{i \cdot}$ are generated by a Gaussian AR(1) process with mean 0, standard deviation 1, and correlation .9.  Each panel corresponds to a simulated $2 \times 100$ matrix $\bX$.}
\label{suppfig1}
\end{center}
\end{figure}

\setcounter{figure}{1} \renewcommand{\thefigure}{A.\arabic{figure}} 
\begin{figure}[ht]
\begin{center}
\includegraphics[scale = .6]{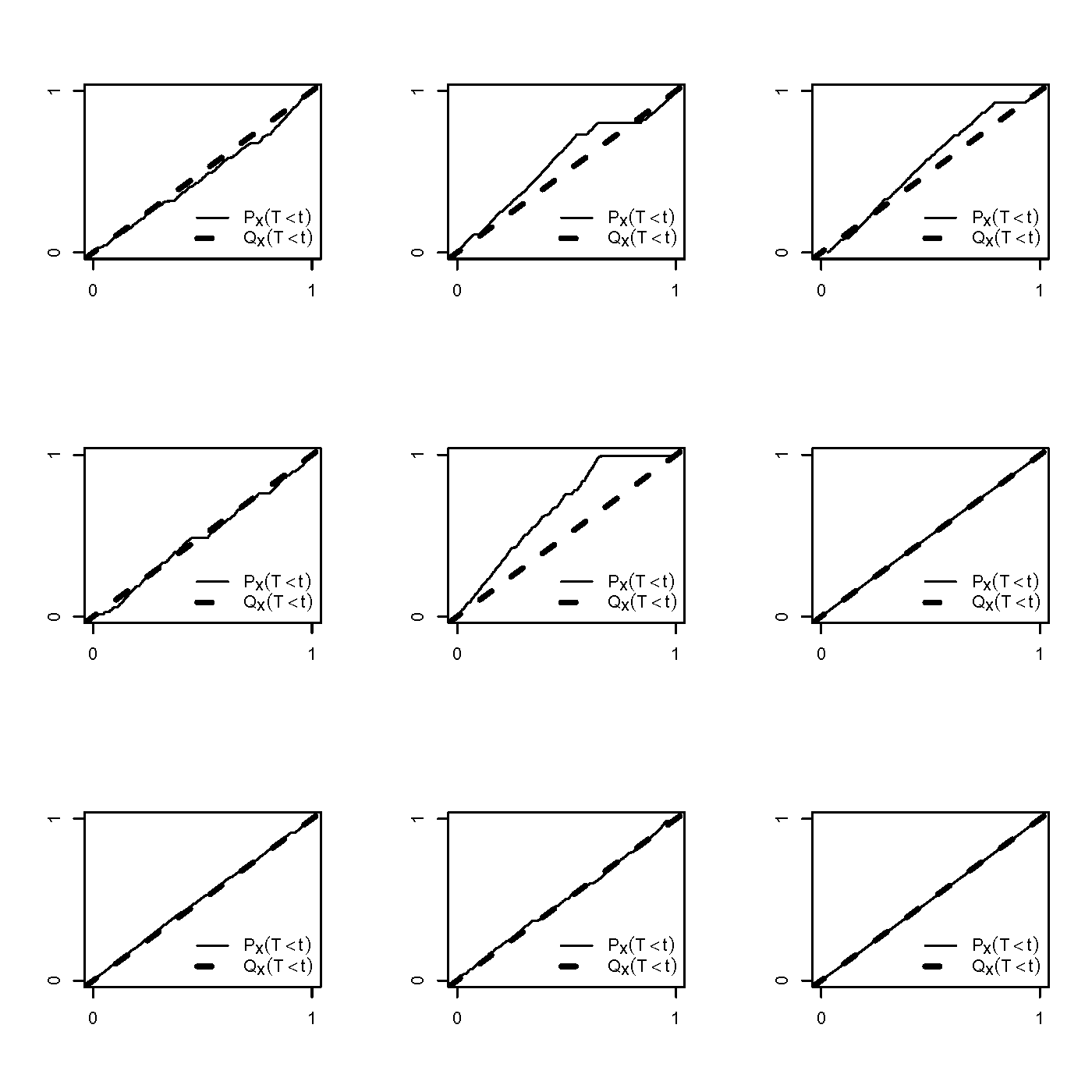}
\caption{Illustration of Resampling Distributions.  Empirical cumulative distribution functions for simulated matrices $\bX$ in which independent rows $\bX_{i \cdot}$ are generated by a Gaussian AR(1) process with mean 0, standard deviation 1, and correlation .9.  Each panel corresponds to a simulated $2 \times 1000$ matrix $\bX$.}
\label{suppfig2}
\end{center}
\end{figure}

\setcounter{figure}{2} \renewcommand{\thefigure}{A.\arabic{figure}} 
\begin{figure}[ht]
\begin{center}
\includegraphics[scale = .6]{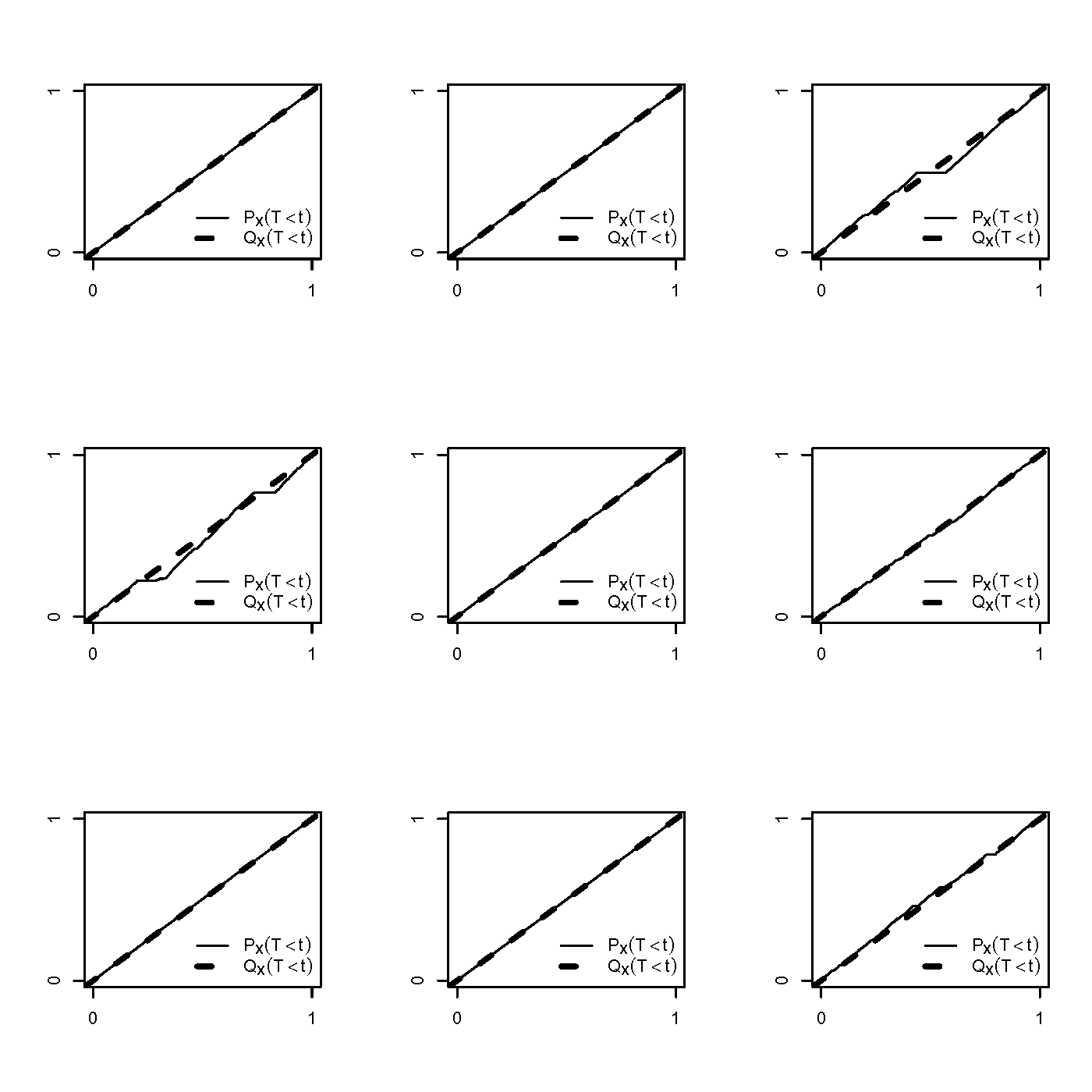}
\caption{Illustration of Resampling Distributions.  Empirical cumulative distribution functions for simulated matrices $\bX$ in which independent rows $\bX_{i \cdot}$ are generated by a Gaussian AR(1) process with mean 0, standard deviation 1, and correlation .9.  Each panel corresponds to a simulated $2 \times 10000$ matrix $\bX$.}
\label{suppfig3}
\end{center}
\end{figure}

\setcounter{figure}{3} \renewcommand{\thefigure}{A.\arabic{figure}} 
\begin{figure}[ht]
\begin{center}
\includegraphics[scale = .6]{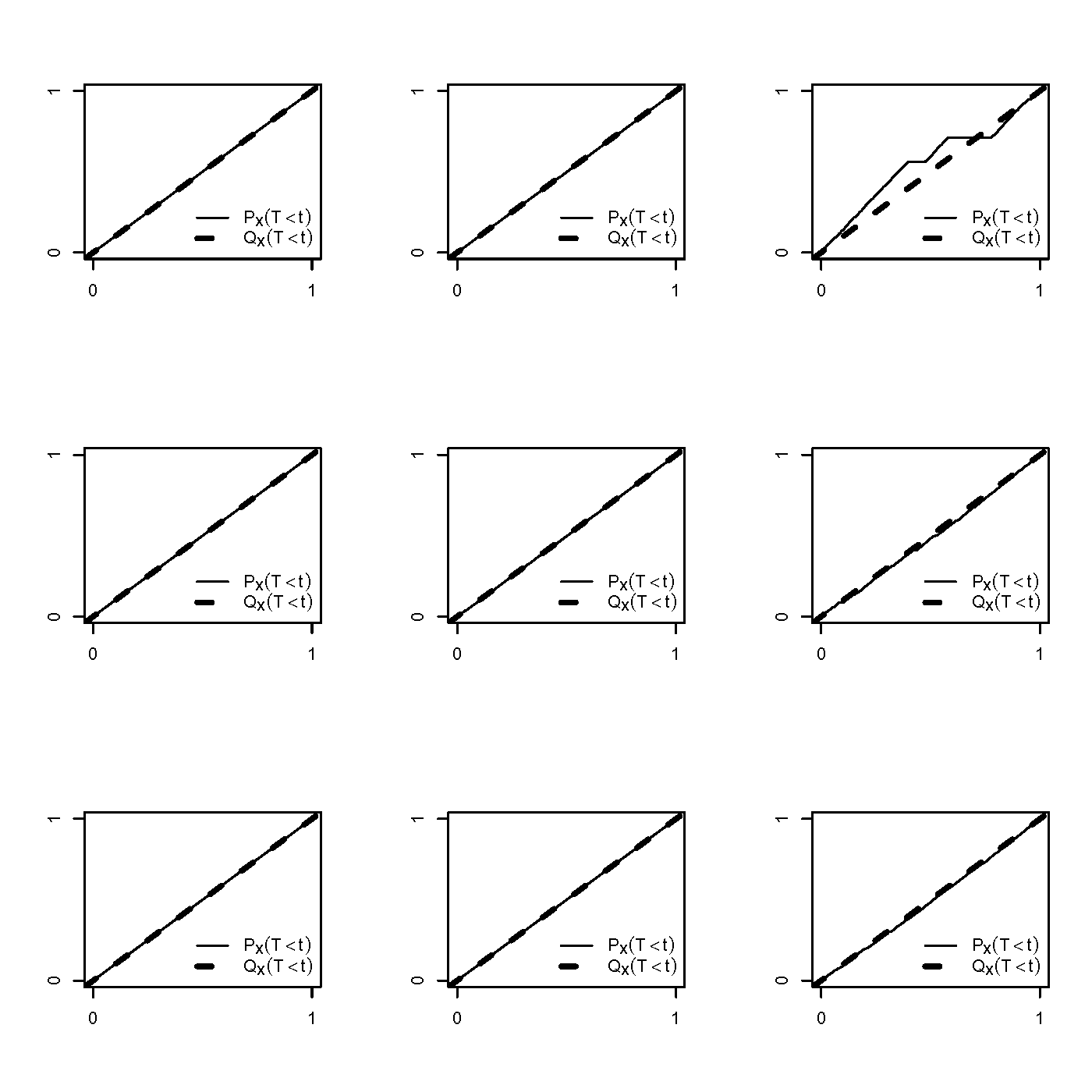}
\caption{Illustration of Resampling Distributions.  Empirical cumulative distribution functions for simulated matrices $\bX$ in which independent rows $\bX_{i \cdot}$ are generated by a Gaussian AR(1) process with mean 0, standard deviation 1, and correlation .9.  Each panel corresponds to a simulated $2 \times 100000$ matrix $\bX$.}
\label{suppfig4}
\end{center}
\end{figure}

\setcounter{figure}{4} \renewcommand{\thefigure}{A.\arabic{figure}} 
\begin{figure}[ht]
\begin{center}
\includegraphics[scale = .6]{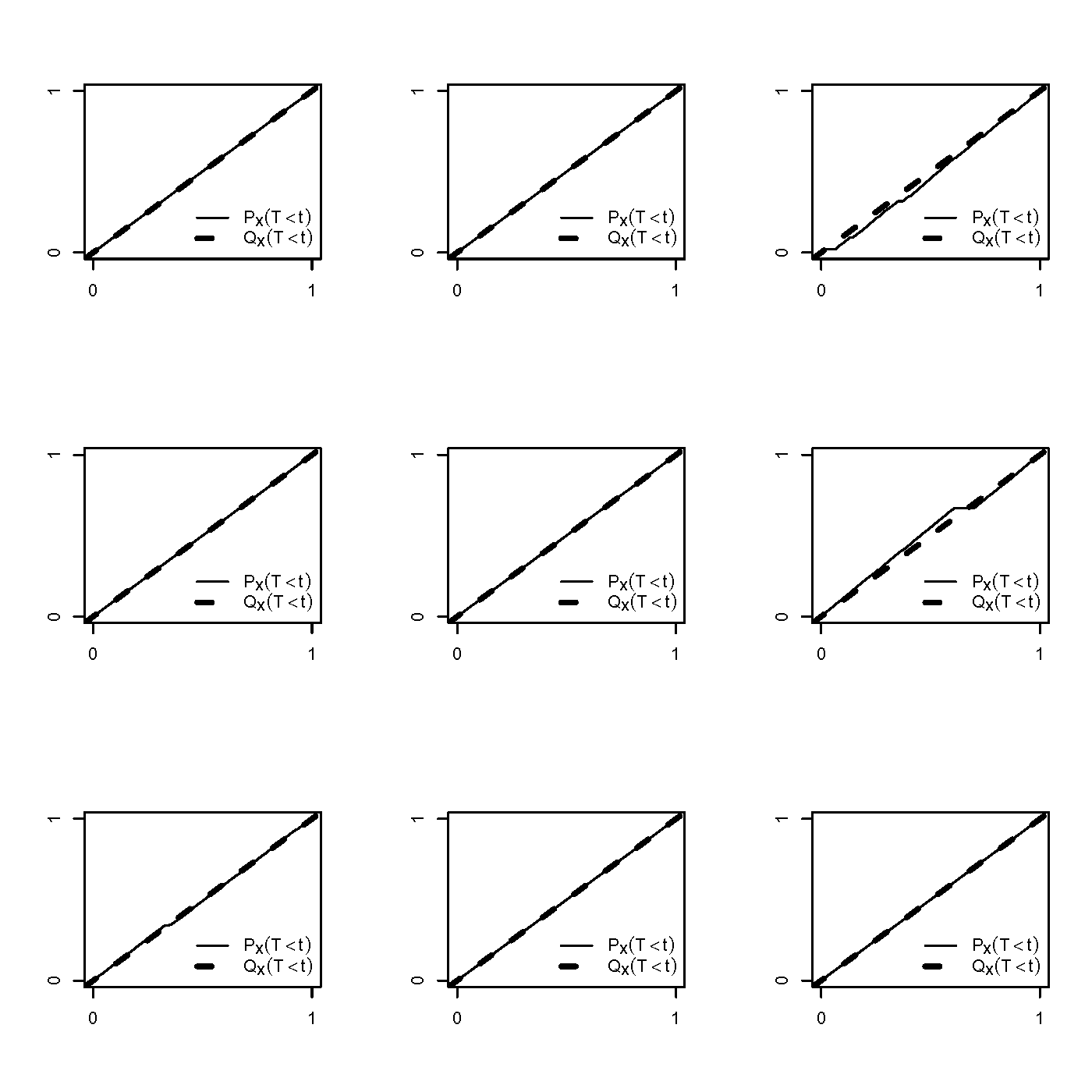}
\caption{Illustration of Resampling Distributions.  Empirical cumulative distribution functions for simulated matrices $\bX$ in which independent rows $\bX_{i \cdot}$ are generated by a Gaussian AR(1) process with mean 0, standard deviation 1, and correlation .9.  Each panel corresponds to a simulated $2 \times 500000$ matrix $\bX$.}
\label{suppfig5}
\end{center}
\end{figure}
\end{document}